\documentclass[12pt, reqno]{amsart}
\usepackage{amssymb,amsthm,amsmath}
\usepackage[numbers,sort&compress]{natbib}
\usepackage{amssymb,amsmath}
\usepackage{amsfonts}
\usepackage{mathrsfs}
\usepackage{latexsym}
\usepackage{amssymb}
\usepackage{amsthm}
\usepackage{indentfirst}
\usepackage{colortbl}
\let\oldsection\section
\renewcommand\section{\setcounter{equation}{0}\oldsection}

\newtheorem{theorem}{Theorem}[section]

\newtheorem{proposition}{Proposition}[section]

\newtheorem{remark}{Remark}[section]

\allowdisplaybreaks
\begin{document}

\title[Blow up of compressible Navier-Stokes equations]{Finite time blow up of compressible Navier-Stokes equations on half space or outside a fixed ball}

\author{Dongfen~Bian}
\address[Dongfen~Bian]{ School of Mathematics and Statistics, Beijing Institute of Technology, Beijing 100081, China; Division of Applied Mathematics, Brown University, Providence, Rhode Island 02912, USA}
\email{dongfen\_bian@brown.edu; biandongfen@bit.edu.cn}

%\author{Chongsheng~Cao}
%\address[Chongsheng~Cao]{Department of Mathematics, Florida International University, University Park, Miami, FL 33199, USA}
%\email{caoc@fiu.edu}

\author{Jinkai~Li}
\address[Jinkai~Li]{South China Research Center for Applied Mathematics and
                Interdisciplinary Studies, South China Normal University,
                Zhong Shan Avenue West 55, Tianhe District, Guangzhou
                510631, China}
\email{jklimath@m.scnu.edu.cn; jklimath@gmail.com}

%\author{Edriss~S.~Titi}
%\address[Edriss~S.~Titi]{
%Department of Mathematics, Texas A\&M University, 3368 TAMU, College Station, TX 77843-3368, USA. ALSO, Department of Computer Science and Applied Mathematics, Weizmann Institute of Science, Rehovot 76100, Israel.}
%\email{titi@math.tamu.edu and edriss.titi@weizmann.ac.il}

%\author{Zhouping Xin}
%\address[Zhouping Xin]{
%The Institute of Mathematical Sciences, The Chinese University of Hong Kong,\\ Hong Kong, P.R.China}
%\email{zpxin@ims.cuhk.edu.hk}


%%% ----------------------------------------------------------------------

\begin{abstract}
In this paper, we consider the initial-boundary value problem to the
compressible Navier-Stokes equations for ideal gases without heat
conduction in the half
space or outside a fixed ball in $\mathbb R^N$, with $N\geq1$.
We prove that any classical solutions $(\rho, u, \theta)$, in the class $C^1([0,T]; H^m(\Omega))$, $m>[\frac N2]+2$,
with bounded from below initial entropy and compactly supported initial density,
which allows to touch the physical boundary, must blow-up in finite time, as long as the initial mass is positive. This paper extends the classical reault by Xin [CPAM, 1998], in which the Cauchy probelm is considered, to the case that with physical boundary.
\end{abstract}

%%% ----------------------------------------------------------------------
\maketitle

%\tableofcontents
\allowdisplaybreaks

\section{Introduction}
\label{sec1}
The compressible Navier-Stokes equations
for idea gases on a domain $\Omega\subseteq\mathbb R^N$, $N\geq1$, read as
\begin{eqnarray}
  &&\partial_t\rho+\text{div}\,(\rho u)=0,\label{cns1}\\
  &&\partial_t(\rho u)+\text{div}\,(\rho u\otimes u)-\text{div}\,S+\nabla p=0,\label{cns2}\\
  &&\partial_t(\rho E)+\text{div}\,(u(\rho E+p))=\text{div}\,q+\text{div}\,(S\cdot u),\label{cns3}
\end{eqnarray}
where the unknowns are the density $\rho\geq0$, the velocity $u\in\mathbb R^N$, and the specific total energy $E\geq0$, with $E=\frac{|u|^2}{2}+e$, and $e$ the specific internal energy. The stress tensor $S$ is given by
$$
S=\mu(\nabla u+(\nabla u)^T)+\lambda\text{div}\,uI,
$$
with two constant Lam\'e viscosity coefficients $\mu$ and $\lambda$ satisfying
$$
\mu\geq0,\quad 2\mu+N\lambda\geq0.
$$
The heat flux $q$ is given by
$q=\kappa\nabla\theta,$
for some nonnegative constant coefficient $k$.
Recalling that we consider the ideal gases, the state equations are
\begin{equation}\label{STAEQ}
e=c_v\theta,\quad p=R\rho\theta,  \quad p=Ae^{s/c_v}\rho^\gamma,
\end{equation}
where $s$ is the entropy, $c_v$, $R$, $A$ and $\gamma>1$ are positive
constants, with $c_v=\frac{R}{\gamma-1}$.

In the absence of vacuum, i.e.\,the density is away from zero, local well-posedness of classical solutions in the H\"older spaces to the compressible Navier-Stokes equations was established
by Itaya \cite{ITAYA} and Tani \cite{TANI}, while the global well-posedness of classical solutions in the Sobolev spaces was firstly established by Matsummura and Nishida \cite{MATSU1,MATSU2}, under the condition that $\|(\rho_0-\bar\rho,u_0,\theta_0-\bar\theta)\|_{H^m}$ is suitably small, where $m>[\frac N2]+2$, and
$\bar\rho$ and $\bar\theta$ are two positive constants.
In the presence of vacuum, it was first proved by Lions \cite{Lions2} the global existence of weak solutions to the isentropic compressible
Navier-Stokes equations (i.e.\,system (\ref{cns1})--(\ref{cns2}) by setting $p=A\rho^\gamma$), with $\gamma\geq\frac{3N}{N+2}$, $N=2,3$. His result
was later extended by Feireisl, Novotny and Pezeltova \cite{Feireisl1} to
the case $\gamma>\frac{3}{2}$, and by Jiang and Zhang \cite{Jiang4} to the
case $\gamma>1$ for axisymmetric solutions. Concerning the full compressible
Navier-Stokes equations, only the global existence of the
so called variational
solutions was proven by Feireisl \cite{Feireisl4,Feireisl5}, where the energy
equation was satisfied only in the sense of
inequality. Local well-posedness of strong solutions to the full compressible Navier-Stokes equations was established by Cho and Kim \cite{CK}; however,
it should be noted that the strong solutions established in \cite{CK} have
no information on the entropy, and in particular it is not known if the
corresponding entropy is bounded or not.

A natural question is whether the classical solutions to the compressible Navier-Stokes equations exist globally or not, when the
initial vacuum is allowed. It was first proved by Xin \cite{XinBlowup}
that smooth solutions, with nontrivial and compactly supported initial density, to any dimensional full compressible Navier-Stokes equations
without heat conduction
or one dimensional isentropic compressible Navier-Stokes equations, must
blow up in finite time. Xin's blow up result
was later generalized by Cho and Jin \cite{ChoJin}, and Tan and Wang \cite{TanWang} to the case with heat conduction, and by Rozanova \cite{Rozanova} to the case of rapidly decreasing solutions.
Moreover, it was shown in a recent paper by Xin
and Yan \cite{XINYAN} that the blow up result may still hold
without the assumptions of compactly supported initial density
or rapidly decreasing of the solutions; they proved that
the blow up for classical solutions occurs in finite time, as long as
the initial density is not identically equal to zero, on a bounded open
set surrounded by vacuum region. Finally, if we focus on the
radially symmetric solutions,
then the finite time blow up result also holds for the two dimensional
isentropic or isothermal compressible Navier-Stokes equations, see Luo \cite{LUO},
and Du, Li and Zhang \cite{DuLiZhang}. However, there is a somewhat
surprising result by Huang,
Li and Xin \cite{HLX1}, where they proved the global well-posedness of
classical solutions to the three dimensional isentropic compressible
Navier-Stokes equations, with initial data of small energy but
allowed to have vacuum and even compactly supported initial density.

Note that in all the papers \cite{XinBlowup,TanWang,ChoJin,Rozanova,LUO,DuLiZhang},
concerning the finite time blow up of classical
solutions to the compressible Navier-Stokes equations, the Cauchy problem
was considered, in other words, the domain under consideration has no
physical boundary. In \cite{XINYAN}, the initial-boundary value
problem was also considered, and thus the physical boundary was allowed;
however, since the additional assumption imposed on the initial data in \cite{XINYAN} prevent the isolated mass group
from touching the boundary, it was essentially reduced to the
case without any physical boundary. In view of the finite time blow up
results in the above mentioned papers, the
remaining question is if the classical
solutions to the compressible Naiver-Stokes equations still blow up in finite
time in the presence of physical boundary. We will partially answer this question. Precisely, we will prove
that if the domain $\Omega$ under consideration
is either the half space $\mathbb R^N_+$ or the exterior
domain $\mathbb R^N\setminus B_{r_0}$, then classical solutions to the
compressible Navier-Stokes equations must blow up in finite time, as long
as the initial mass is positive and the initial density is compactly
supported in $\overline\Omega$.

In this paper, we consider the compressible Navier-Stokes equations without heat conduction, in other words, we consider the following system
\begin{eqnarray}
  &&\partial_t\rho+\text{div}\,(\rho u)=0,\label{CNS1}\\
  &&\partial_t(\rho u)+\text{div}\,(\rho u\otimes u)-\text{div}\,S+\nabla p=0,\label{CNS2}\\
  &&\partial_t(\rho E)+\text{div}\,(u(\rho E+p))=\text{div}\,(S\cdot u).\label{CNS3}
\end{eqnarray}
We always suppose that the viscosity coefficients $\mu$ and $\lambda$ satisfy
\begin{equation}\label{COEFF}
  \mu>0,\quad 2\mu+N\lambda>0.
\end{equation}

We consider the initial-boundary value problems to system (\ref{CNS1})--(\ref{CNS3}), on the half space or outside a fixed ball (without loss of generality, we can suppose that the fixed ball is centered at the origin). Hence, the domain $\Omega$ under consideration is taken as one of the following two cases:

(i) $\Omega=\mathbb R^N_+=\{(x_1,\cdots,x_N)|x_i\in\mathbb R, 1\leq i\leq N-1, x_N>0\}$,

(ii) $\Omega=\mathbb R^N\setminus \overline B_{r_0} =\{x\in\mathbb R^N||x|>r_0\}$, for some positive number $r_0$.
\\
We complement system (\ref{CNS1})--(\ref{CNS3}) with the following
boundary condition
\begin{equation}
  \label{BC}
  u|_{\partial\Omega}=0,\quad (u(x,t),\theta(x,t))\rightarrow0,\mbox{ as }x\rightarrow\infty,
\end{equation}
while the initial condition
reads as
\begin{equation}\label{IC}
  (\rho, u, \theta)|_{t=0}=(\rho_0, u_0,\theta_0)\in H^m(\Omega),
\end{equation}
for some integer $m>[\frac N2]+2$.

Recalling the state equation $p=Ae^{s/c_v}\rho^\gamma$ in (\ref{STAEQ}), it is
natural for us to assume the following compatibility condition on the initial data
\begin{equation}
  \label{COMP}
  \theta_0(x)>0,\mbox{ for any }x\in\mathcal O_+:=\{x\in\Omega|\rho_0(x)>0\}.
\end{equation}
As a result, in the non-vacuum region, by the state equations in (\ref{STAEQ}), the initial data of the entropy $s$ on $\mathcal O_+$ is well-defined as
\begin{equation}\label{ICS}
  s_0(x)=c_v\log\left(\frac{R}{A}\theta_0(x)\rho_0^{1-\gamma}(x)\right), \mbox{ for }x\in\mathcal O_+.
\end{equation}

%Note that if $\rho\not=0$ and $\theta\not=0$, then
%the entropy $s$ can be expressed in term of $\rho$ and $\theta$ as
%$$
%s=c_v(\log\frac RA+\log\theta-(\gamma-1)\log\rho).
%$$
%However, this relation cannot be extend to the vacuum region, i.e.\,the region
%where the density $\rho$ vanishes.

%\begin{definition}
%  Given a positive time $T\in(0,\infty)$. A triple $(\rho, u, \theta)$ is called a classical solution to system (\ref{CNS1})--(\ref{CNS3}), subject
%  to (\ref{BC})--(\ref{IC}), on $\Omega\times(0,T)$, if the following three
%  hold:
%
%  (i) (Regularities) $(\rho, u, \theta)\in C^1([0,T]; H^m(\Omega)),$ for some integer $m>[\frac N2]+2$;
%
%  (ii) (Entropy) the entropy is well-defined almost everywhere, i.e.\,the relationship $\theta=\frac{A}{R}e^{s/c_v}\rho^{\gamma-1}$ holds, a.e.\,on $\Omega\times(0,T)$, for some finite valued function $s$ on $\Omega\times(0,T)$;
%
%  (iii) $(\rho, u, \theta)$ satisfies equations (\ref{CNS1})--(\ref{CNS3}),
%  and fulfills the boundary and initial conditions (\ref{BC})--(\ref{IC}).
%\end{definition}

We have the following theorem on the blow up of classical solutions to system (\ref{CNS1})--(\ref{CNS3}), subject to (\ref{BC})--(\ref{IC}).

\begin{theorem}\label{thm1}
Suppose that $(\rho_0, u_0, \theta_0)\in H^m(\Omega)$, for some $m>[\frac N2]+2$, satisfying the compatibility condition (\ref{COMP}), and
$$
supp~\rho_0\subseteq B_{R_0}, \quad \int_\Omega\rho_0(x)dx>0,
$$
for some positive number $R_0$ (with $R_0>r_0$, for the case $\Omega=\mathbb R^N\setminus \overline B_{r_0}$).
Let $s_0$ be the function defined by (\ref{ICS}) on $\mathcal O_+$, satisfying \begin{equation}\label{ASSU}
\inf_{x\in\mathcal O_+}s_0(x)\geq\underline s_0,
\end{equation}
for some constant $\underline s_0$.

%Then, any classical solution $(\rho, u, \theta)$, in the class $C^1([0,T]; H^m(\Omega))$, to the compressible Navier-Stokes equations (\ref{CNS1})--(\ref{CNS3}), subject to the boundary and initial conditions (\ref{BC})--(\ref{IC}), must blow up in finite time.

{Then, the compressible Navier-Stokes equations (\ref{CNS1})--(\ref{CNS3}), subject to (\ref{BC})--(\ref{IC}), either do not have
any local classical solution $(\rho, u, \theta)$ in the class $C^1([0,T]; H^m(\Omega))$, or otherwise this local solution must blow up in finite time.}
\end{theorem}

For a special case, the assumption (\ref{ASSU}) in the above theorem can
be removed, and in fact we have the following:

\begin{theorem}
  \label{thm2}
If $N=1$, then the result in Theorem \ref{thm1} still holds without the assumption (\ref{ASSU}).
\end{theorem}

{\begin{remark}
It should be pointed out that, same as in \cite{XinBlowup,XINYAN}, the existence of local solution $(\rho, u, \theta)$ in the class $C^1([0,T]; H^m(\Omega))$ to system (\ref{CNS1})--(\ref{CNS3}), subject to (\ref{BC})--(\ref{IC}), is still open. If following the arguments
in \cite{CHOCLASS,CK}, one can obtain a unique solution $(\rho, u, p)$ ($p$ is chosen as an unknown) in the class
$$
(\rho, p)\in C^1([0.T]; H^m(\Omega)), \quad u\in C^1([0,T]; \dot{H}^m(\Omega)),
$$
here $\dot{H}^m(\Omega):=\{f\in L_{loc}(\Omega)|\nabla^\alpha f\in L^2(\Omega), 1\leq|\alpha|\leq m\}$, which unfortunately does not
meet the requirements on $(\rho, u, \theta)$ in Theorem \ref{thm1} and Theorem \ref{thm2}. The existence of solution $(\rho, u, \theta)$ in the class $C([0,T]; H^1(\mathbb R))$ (which can be further strengthened in $C^1([0,T]; H^m(\mathbb R))$ if putting more regularity assumptions and compatibility conditions
on the initial data) to the Cauchy problem of (\ref{CNS1})--(\ref{CNS3}), in the presence of vacuum at the far field only, has
recently been proved by the second author \cite{LIXIN17}. Unfortunately, since the crucial assumption (1.12) there does not hold if $\rho_0$ is compactly supported, the argument in \cite{LIXIN17} does not leads to the desired existence of solution
 $(\rho, u, \theta)$ required in Theorem \ref{thm1} and Theorem \ref{thm2} either.
\end{remark}}

Some comments on the proofs of Theorem \ref{thm1} and Theorem \ref{thm2}
are stated as follows. Note that the main ingredients of the proofs in
\cite{XinBlowup,TanWang,ChoJin,Rozanova,LUO,DuLiZhang,XINYAN} are
multiplying the transport equation by $|x|^2$, and correspondingly
multiplying the momentum equations by $x$, where the key observation is
that the term $\int_\Omega\text{div}S\cdot xdx$ vanishes, if either the
domain under consideration has no physical boundary or the isolated mass
group never touches the boundary. Unfortunately, it is not the case
when there is some physical boundary of the domain, and one can not
expect that the isolated mass group will never touch the boundary, even if
it is initially away from the boundary. Therefore, we will always encounter
some boundary integrals coming from $\int_\Omega\text{div}S\cdot xdx$
after integration by parts. To overcome this difficulty, taking the case
$\Omega=\mathbb R^N\setminus B_{r_0}$ as an example, we multiply the
transport equation by some
positive function $f(|x|)$, rather than the very special function $|x|^2$,
and correspondingly multiply the momentum equation by $f'(|x|)\frac{x}{|x|}$,
and encounter the term $\int_\Omega\text{div}Sf'(|x|)\frac{x}{|x|}dx$. To
ensure that this last term vanishes, by integration by parts, it suffices to
ask for $f'(|x|)|_{\partial\Omega}=0$ and $\Delta(f'(|x|)\frac{x}{|x|})=0$ on $\Omega$. The existence of such an auxiliary function $f$ can be easily verified, and consequently one can obtain the finite time
blow up results.

\section{Proofs of the theorems}

Given a velocity field $u\in C^1([0,T]; H^m(\Omega))$,
for some $m>[\frac N2]+2$, with
$u=0$ on $\partial\Omega$. Denote by $X(t; x)$ the particle path,
which goes along the velocity field $u$ and starts from $x\in \Omega$
at time zero:
\begin{equation}\label{particle}
  \left\{
  \begin{array}{l}
  \frac{d}{dt}X(t; x)=u(X(t;x),t),\quad t\in(0,T),\\
  X(0;x)=x.
  \end{array}
  \right.
\end{equation}
For any subset $K\subseteq \Omega$, for simplicity we denote
$$
X(t; K):=\{y=X(t;x)|x\in K\}.
$$
By the Sobolev embedding, one has $u\in C^1([0,T]; C^2(\overline\Omega))$,
and thus, by the standard
existence and uniqueness results for ordinary differential equations,
the particle pathes are well-defined, and
different particle pathes never meet each other. Moreover, at each time, any point $y\in\Omega$ can be reached by
some particle path, in other words, one has $X(t; \Omega)= \Omega$. Using these facts, one can easily verify that
$$
[X(t;K)]^c\cap \Omega=X(t; K^c\cap \Omega),
\quad
X(t; K_1)\cup X(t; K_2)=X(t; K_1\cup K_2),
$$
for any subsets $K$, $K_1$ and $K_2$ of $\Omega$. These facts will be used
later without any further mentions.

Recalling the expression of the stress tensor $S$, we have
\begin{align*}
  S:\nabla u=&\mu(\nabla u+(\nabla u)^T):\nabla u+\lambda(\text{div}\,u)^2\\
  =&\frac\mu2|\nabla u+(\nabla u)^T|^2+\lambda(\text{div}\,u)^2,
\end{align*}
which simply implies that
\begin{equation}\label{nonneg1}
 S:\nabla u\geq\frac\mu2|\nabla u+(\nabla u)^T|^2\geq0,
\end{equation}
if $\lambda\geq0$. While if $\lambda<0$, by transforming $S:\nabla u$ as
\begin{align*}
  S:\nabla u
  =&\frac\mu2\sum_{ {i\not=j}}(\partial_iu_j+\partial_ju_i)^2+(2\mu+N\lambda)\sum_{i=1}^N(\partial_iu_i)^2\\
  &-(N-1)\lambda\sum_{i=1}^N(\partial_iu_i)^2+\lambda\sum_{ {i\not=j}}\partial_iu_i\partial_ju_j\\
  =&\frac\mu2\sum_{ {i\not=j}}(\partial_iu_j+\partial_ju_i)^2+(2\mu+N\lambda)\sum_{i=1}^N(\partial_iu_i)^2 -\frac\lambda2\sum_{ {i\not=j}}(\partial_iu_i-\partial_ju_j)^2,
\end{align*}
and recalling (\ref{COEFF}), we still have
\begin{equation}\label{nonneg2}
  S:\nabla u\geq \frac\mu2\sum_{ {i\not=j}}(\partial_iu_j+\partial_ju_i)^2+(2\mu+N\lambda)\sum_{i=1}^N(\partial_iu_i)^2\geq0.
\end{equation}

Some preparations are required before proving our main results, that is the
following two propositions.

\begin{proposition}\label{prop1}
Let $(\rho, u, \theta)\in C^1([0,T]; H^m(\Omega))$, with $m>[\frac N2]+2$, be a classical
solution to system (\ref{CNS1})--(\ref{CNS3}), subject to (\ref{BC})--(\ref{IC}). Suppose that
$$
supp~\rho_0\subseteq B_{R_0},
$$
for some positive number $R_0$ (with $R_0>r_0$, if $\Omega=\mathbb R^N\setminus \overline B_{r_0}$).
Then, we have
$$
supp~\rho(\cdot,t)\subseteq B_{R_0},\quad supp~u(\cdot,t)\subseteq B_{R_0}.
$$
\end{proposition}

\begin{proof}
By the definition of $X(t; x)$, and using (\ref{CNS1}), we deduce
\begin{align*}
  \frac{d}{dt}\rho(X(t; x),t)
  =&-\text{div}\,u(X(t;x),t)\rho(X(t;x),t),
\end{align*}
from which, by assumption, and recalling $u\in C^1([0,T]; C^2(\overline\Omega))$, we have
\begin{align}
  \rho(X(t;x),t)
  =&\exp\left\{-\int_0^t\text{div}\,u(X(\tau;x),\tau)d\tau\right\}\rho_0(x)=0,\label{rhoeq}
\end{align}
for any $x\in B_{R_0}^c\cap \Omega$. Hence, one has
\begin{equation}\label{rhovanish}
\rho(\cdot,t)\equiv0,\quad\mbox{ on } X(t;B_{R_0}^c\cap \Omega)=[X(t;B_{R_0}\cap \Omega)]^c.
\end{equation}
Thanks to this, it follows from equations (\ref{CNS2}) and (\ref{CNS3}) that
$$
\text{div}\,S=\text{div}\,(S\cdot u)=0,\quad \mbox{ on }[X(t;B_{R_0}\cap \Omega)]^c\times\{t\},
$$
and thus
\begin{equation*}
  S:\nabla u=\text{div}\,(S\cdot u)-\text{div}\,S\cdot u=0,\quad \mbox{ on }[X(t;B_{R_0}\cap \Omega)]^c\times\{t\}.
\end{equation*}
As a result, it follows from (\ref{nonneg1}) and (\ref{nonneg2}) that
\begin{equation*}
  \partial_iu_j+\partial_ju_i=0,\quad 1\leq i,j\leq N, \quad\mbox{ on }[X(t;B_{R_0}\cap \Omega)]^c\times\{t\}.
\end{equation*}
This and the assumption $u\in C^1([0,T]; H^m(\Omega))$ imply
\begin{equation}\label{uvanish}
u(\cdot,t)\equiv0, \quad\mbox{ on }[X(t;B_{R_0}\cap \Omega)]^c=X(t; B_{R_0}^c\cap \Omega).
\end{equation}
Thus, for any $x\in B_{R_0}^c\cap \Omega$, one has
$$
\frac{d}{dt}X(t; x)=u(X(t;x),t)=0,
$$
which implies $X(t; x)=x$, for any $x\in B_{R_0}^c\cap \Omega$. Therefore, we have
$$[X(t;B_{R_0}\cap \Omega)]^c=X(t; B_{R_0}^c\cap \Omega)=B_{R_0}^c\cap \Omega,$$
and consequently, the conclusion follows from (\ref{rhovanish}) and (\ref{uvanish}).
\end{proof}

\begin{proposition}\label{prop2}
Let $(\rho, u, \theta)\in C^1([0,T]; H^m(\Omega))$, with $m>[\frac N2]+2$, be a classical
solution to system (\ref{CNS1})--(\ref{CNS3}), subject to (\ref{BC})--(\ref{IC}). Suppose that the compatibility condition (\ref{COMP}) holds.
Then, we have
 $$
 \rho(x,t)>0\mbox{ and}\quad\theta(x,t)>0,
 $$
for any $x\in\mathcal O_+(t):=X(t;\mathcal O_+)$, and
$$
\inf_{x\in\mathcal O_+(t)}\log(\theta\rho^{1-\gamma})\geq\inf_{x\in\mathcal O_+}\log(\theta_0\rho_0^{1-\gamma}).
$$
\end{proposition}

\begin{proof}
For any $x\in\mathcal O_+$, by equation (\ref{rhoeq}), we have $\rho(X(t;x),t)>0$. Hence, $\rho(x,t)>0$, for any $x\in\mathcal O_+(t)$. To prove the positivity of $\theta$ on $\mathcal O_+(t)$, we need to derive the equation for $\theta$.
Using equation (\ref{CNS1}) and the state equation $e=c_v\theta$, it follows from equation (\ref{CNS3}) that
\begin{equation*}
  \rho\left[\partial_t\left(\frac{|u|^2}{2}+c_v\theta\right)+u\cdot\nabla \left(\frac{|u|^2}{2}+c_v\theta\right)\right] +\text{div}(up)=\text{div}(S\cdot u).
\end{equation*}
Multiplying equation (\ref{CNS2}) by $u$, and using equation (\ref{CNS1}) yields
\begin{equation*}
  \rho\left[\partial_t\left(\frac{|u|^2}{2}\right)+u\cdot\nabla \left(\frac{|u|^2}{2}\right)\right]-\text{div}(S\cdot u)+S:\nabla u+\text{div}(up)-\text{div}up=0.
\end{equation*}
Subtracting the previous two equations, recalling the state equation $p=R\rho\theta$, one obtains
\begin{equation}\label{thetaeq}
  c_v\rho(\partial_t\theta+u\cdot\nabla\theta)+R\rho\text{div}u\,\theta=S:\nabla u.
\end{equation}
Hence, recalling the nonnegativity of $S:\nabla u$, see (\ref{nonneg1}) and (\ref{nonneg2}), we deduce
\begin{align*}
  \frac{d}{dt}\theta(X(t; x), t)=&\frac{1}{c_v}\frac{1}{\rho}(S:\nabla u-R\rho\text{div}u\,\theta)|_{(X(t;x),t)}\\
  =&\frac{1}{c_v}\left(\frac{ S}{\rho}:\nabla u-R\text{div}u\,\theta\right)\bigg|_{(X(t;x),t)}\geq (1-\gamma)\text{div}u\,\theta|_{(X(t;x),t)},
\end{align*}
from which, recalling that $u\in C^1([0,T]; C^2(\overline\Omega))$, one obtains $\theta(X(t; x),t)>0$, for any $x\in\mathcal O_+$. Therefore,
we have $\theta(x,t)>0$, for any $x\in\mathcal O_+(t)$. This proves the first conclusion.

Now, let us prove the second conclusion. Take arbitrary $x\in\mathcal O_+$,  then $\rho(X(t;x),t)>0$ and $\theta(X(t;x),t)>0$. Set $l(t)=(\theta\rho^{1-\gamma})|_{(X(t;x),t)}$, then it follows from equations (\ref{CNS1}) and (\ref{thetaeq}) that
\begin{align*}
  l'(t)=&(1-\gamma)(\theta\rho^{-\gamma})|_{(X(t;x),t)}\frac{d}{dt}\rho(X(t;x),t)\\
  &+\rho^{1-\gamma}(X(t;x),t)\frac{d}{dt}\theta (X(t;x),t)\\
  =&(1-\gamma)(\theta\rho^{-\gamma})|_{(X(t;x),t)}(-\text{div}u\,\rho)|_{(X(t;x),t)}\\
  &+\rho^{-\gamma}(X(t;x),t)\frac{1}{c_v}(S:\nabla u-R\rho\text{div}u\,\theta)|_{(X(t;x),t)}\\
  =&\frac{1}{c_v}\rho^{-\gamma}S:\nabla u|_{(X(t;x),t)}\geq0.
\end{align*}
Hence, one has
$$
(\theta\rho^{1-\gamma})|_{(X(t;x),t)}=l(t)\geq l(0)=\theta_0(x)\rho_0^{1-\gamma}(x),\quad\forall x\in\mathcal O_+,
$$
from which, by taking the logarithm to both sides of the above inequality yields
$$
\log(\theta\rho^{1-\gamma})|_{(X(t;x),t)}\geq \log(\theta_0(x)\rho_0^{1-\gamma}(x))\geq\inf_{x\in\mathcal O_+}\log(\theta_0\rho_0^{1-\gamma}),\quad\forall x\in\mathcal O_+.
$$
Therefore, we have
$$
\log(\theta(x,t)\rho^{1-\gamma}(x,t))\geq\inf_{x\in\mathcal O_+}\log(\theta_0\rho_0^{1-\gamma}),\quad\forall x\in\mathcal O_+(t),
$$
proving the second conclusion.
\end{proof}

We are now ready to prove the main results.

\begin{proof}[\emph{\textbf{Proof of Theorem \ref{thm1}}}] \textbf{Case I: $\Omega=\mathbb R^N\setminus B_{r_0}$.}
Define two radially symmetric functions $f$ and $g$ on $\mathbb R^N\setminus\{0\}$ as
\begin{equation*}
  f(|x|)=\left\{
  \begin{array}{lr}\vspace{2mm}
  \frac{1}{2r_0}(|x|-r_0)^2,&N=1,\\
  \vspace{2mm}\frac{|x|^2}{2r_0^2}-\log|x|+\log r_0,&N=2,\\
\frac{|x|^2}{2r_0^N}+\frac{1}{N-2}\frac{1}{|x|^{N-2}},&N\geq3,
  \end{array}
  \right.
\end{equation*}
and
$$
g(|x|)=\frac{1}{r_0^N}-\frac{1}{|x|^N}.
$$
Then, one can check that
\begin{eqnarray*}
    &&f(|x|), g(|x|), f'(|x|), g'(|x|)>0,\quad\mbox{ for }|x|>r_0,\\
    &&\nabla f(|x|)=g(|x|)x, \quad\nabla(g(|x|)x)=g(|x|)I+\frac{g'(|x|)}{|x|}x\otimes x,\quad\mbox{ for }|x|>0,\\
    &&\text{div}\,(g(|x|)x)=\frac{N}{r_0^N},\quad\Delta (g(|x|)x)=\nabla\text{div}\,(g(|x|)x)=0,\quad\mbox{ for }|x|>0.
\end{eqnarray*}
Define $m_0, m_1$ and $m_2$ as
\begin{eqnarray*}
  m_0=\int_\Omega\rho_0 dx,\quad m_1=\int_\Omega\rho_0f(|x|)dx,\quad m_2= \int_\Omega g(|x|)\rho_0u_0\cdot xdx.
\end{eqnarray*}
Then, by assumption, we have $m_0>0$.
  Note that, using Proposition \ref{prop1} and the boundary condition (\ref{BC}), we have
  \begin{equation}\label{NNN1}
    \rho|_{\partial B_{R_0}}=p|_{\partial B_{R_0}}=0,\quad u|_{\partial B_{r_0}}=u|_{\partial B_{R_0}}=0,\quad\nabla u|_{\partial B_{R_0}}=0.
  \end{equation}

  Using equation (\ref{CNS1}) and the above boundary conditions,
  it follows from integration by parts that
  \begin{align}
    \frac{d}{dt}\int_{B_{R_0}\setminus B_{r_0}}&\rho f(|x|)dx =\int_{B_{R_0}\setminus B_{r_0}}\partial_t\rho f(|x|)dx \nonumber\\
    =&-\int_{B_{R_0}\setminus B_{r_0}}\text{div}(\rho u)f(|x|)dx
    =\int_{B_{R_0}\setminus B_{r_0}}\rho u\cdot\nabla f(|x|)dx \nonumber\\
    =&\int_{B_{R_0}\setminus B_{r_0}}g(|x|)\rho u\cdot xdx. \label{m1}
  \end{align}
  Noticing that $g|_{\partial B_{r_0}}=0$, and recalling the boundary
  conditions (\ref{NNN1}), it follows from equation (\ref{CNS2}) and integration by parts that
  \begin{align}
    &\frac{d}{dt}\int_{B_{R_0}\setminus B_{r_0}}g(|x|)\rho u\cdot xdx=\int_{B_{R_0}\setminus B_{r_0}}g(|x|)\partial_t(\rho u)\cdot xdx\nonumber\\
    =&-\int_{B_{R_0}\setminus B_{r_0}}g(|x|)[\text{div}(\rho u\otimes u)-\mu\Delta u-(\mu+\lambda)\nabla\text{div}\,u+\nabla p]\cdot xdx\nonumber\\
    =&\int_{B_{R_0}\setminus B_{r_0}}[\rho(u\otimes u):\nabla(g(|x|)x)-\mu\nabla u:\nabla(g(|x|)x)\nonumber\\
    &-(\mu+\lambda)\text{div}\,u\,\text{div}( g(|x)x)+p\text{div}(g(|x|)x)]dx\nonumber\\
    =&\int_{B_{R_0}\setminus B_{r_0}}[\mu\Delta(g(|x|)x)+(\mu+\lambda)\nabla\text{div}(g(|x|)x)]\cdot udx \nonumber\\
    &+\int_{B_{R_0}\setminus B_{r_0}}[\rho(u\otimes u):\nabla(g(|x|)x)+p\text{div}(g(|x|)x)]dx\nonumber\\
    =&\int_{B_{R_0}\setminus B_{r_0}}[\rho(u\otimes u):\nabla(g(|x|)x)+p\text{div}(g(|x|)x)]dx\nonumber\\
    =&\int_{B_{R_0}\setminus B_{r_0}} \left[\rho(u\otimes u):\left(g(|x|)I+\frac{g'(|x|)}{|x|}x\otimes x\right)+\frac{N}{r_0^N}p\right]dx\nonumber\\
    =&\int_{B_{R_0}\setminus B_{r_0}} \left[g(|x|)\rho|u|^2+\frac{g'(|x|)}{|x|}\rho|u\cdot x|^2+\frac{N}{r_0^N}p\right]dx\nonumber\\
    \geq&\frac{N}{r_0^N}\int_{B_{R_0}\setminus B_{r_0}}pdx. \label{m2'}
  \end{align}

Recalling (\ref{NNN1}), it follows from equation (\ref{CNS1}) and integration by parts that
\begin{align*}
  \frac{d}{dt}\int_{B_{R_0}\setminus B_{r_0}}\rho dx=\int_{B_{R_0}\setminus B_{r_0}} \partial_t\rho dx=-\int_{B_{R_0}\setminus B_{r_0}}\text{div}(\rho u)dx=0,
\end{align*}
which provides
\begin{equation}
  \label{m0}
  \int_{B_{R_0}\setminus B_{r_0}}\rho(x,t) dx=\int_{B_{R_0}\setminus B_{r_0}}\rho_0 dx=m_0.
\end{equation}

Applying Proposition \ref{prop2}, for any $x\in\mathcal O_+(t):=X(t;\mathcal O_+)$, we have $\rho(x,t)>0$, $\theta(x,t)>0$, and
\begin{align*}
  p(x,t)=&R\rho\theta=R\rho^\gamma\exp\{\log(\theta\rho^{1-\gamma})\}\\
  \geq&R\rho^\gamma\exp\left\{\inf_{y\in\mathcal O_+}\log(\theta_0(y)\rho_0^{1-\gamma}(y))\right\}\\
  =&R\rho^\gamma\exp\left\{\inf_{x\in\mathcal O_+}\left(\frac{s_0(x)}{c_v}+\log\frac AR\right)\right\}\\
  =&R\rho^\gamma\exp\left\{\frac{\underline s_0}{c_v}+\log\frac AR\right\}=Ae^{\underline s_0/c_v}\rho^\gamma(x,t).
\end{align*}
Thanks to the above estimate, and noticing that
$\mathcal O_+(t)=\{x\in\Omega|\rho(x,t)>0\},$ we deduce
\begin{align}
  \int_{B_{R_0}\setminus B_{r_0}}pdx=&\int_{(B_{R_0}\setminus B_{r_0})\cap\mathcal O_+(t)}pdx\geq Ae^{\underline s_0/c_v} \int_{(B_{R_0}\setminus B_{r_0})\cap\mathcal O_+(t)}\rho^\gamma dx\nonumber\\
  =&Ae^{\underline s_0/c_v}\int_{B_{R_0}\setminus B_{r_0}}\rho^\gamma dx.\label{NEW7}
\end{align}
By the H\"older inequality, we have
\begin{align*}
  \int_{B_{R_0}\setminus B_{r_0}}\rho dx\leq\left(\int_{B_{R_0}\setminus B_{r_0}}\rho^\gamma dx\right)^{\frac1\gamma}|B_{R_0}\setminus B_{r_0}|^{1-\frac1\gamma},
\end{align*}
and thus, recalling (\ref{m0}), we have
$$
\int_{B_{R_0}\setminus B_{r_0}}\rho^\gamma dx\geq [\omega_N(R_0^N-r_0^N)]^{1-\gamma}m_0^\gamma,
$$
where $\omega_N$ is the volume of the unit ball in $\mathbb R^N$.
Substituting the above estimate into (\ref{NEW7}) yields
\begin{align}\label{ESTP}
  \int_{B_{R_0}\setminus B_{r_0}}pdx \geq Ae^{\underline s_0/c_v}[\omega_N(R_0^N-r_0^N)]^{1-\gamma}m_0^\gamma.
\end{align}

Thanks to (\ref{ESTP}), it follows from (\ref{m2'}) that
\begin{align*}
  &\int_{B_{R_0}\setminus B_{r_0}}g(|x|)\rho(x,t) u(x,t)\cdot xdx\\
  \geq& \int_{B_{R_0}\setminus B_{r_0}}g(|x|)\rho_0 u_0\cdot xdx+
  ANe^{\underline s_0/c_v}[\omega_N(R_0^N-r_0^N)]^{1-\gamma}\frac{m_0^\gamma}{r_0^N}t\\
  =&m_2+
  ANe^{\underline s_0/c_v}[\omega_N(R_0^N-r_0^N)]^{1-\gamma}\frac{m_0^\gamma}{r_0^N}t.
\end{align*}
With the aid of the above estimate, it follows from (\ref{m1}) that
\begin{align*}
  &\int_{B_{R_0}\setminus B_{r_0}}\rho(x,t) f(|x|)dx\\
  \geq&\int_{B_{R_0}\setminus B_{r_0}}\rho_0f(|x|)dx+m_2t+ANe^{\underline s_0/c_v}[\omega_N(R_0^N-r_0^N)]^{1-\gamma}\frac{m_0^\gamma}{2r_0^N}t^2\nonumber\\
  =&m_1+m_2t+ANe^{\underline s_0/c_v}[\omega_N(R_0^N-r_0^N)]^{1-\gamma}\frac{m_0^\gamma}{2r_0^N}t^2.
\end{align*}
On the other hand, by (\ref{m0}), one has
\begin{equation*}
  \int_{B_{R_0}\setminus B_{r_0}}\rho(x,t) f(|x|)dx\leq f(R_0)\int_{B_{R_0}\setminus B_{r_0}}\rho dx=f(R_0)m_0.
\end{equation*}
Combining the above two estimates, we then obtain
$$
m_1+m_2t+ANe^{\underline s_0/c_v}[\omega_N(R_0^N-r_0^N)]^{1-\gamma}\frac{m_0^\gamma}{2r_0^N}t^2\leq f(R_0)m_0.
$$
which, recalling that $m_0>0$, implies $t\leq T_*$, for some finite time $T_*$. Therefore, $(\rho, u, \theta)$
can not exist for all time. This completes the proof of Case I.

\textbf{Case II: $\Omega=\mathbb R^N_+$.}
Define $M_0, M_1$ and $M_2$ as
$$
M_0=\int_\Omega\rho_0 dx, \quad M_1=\int_\Omega\rho_0u_{0N}x_N dx,\quad M_2= \int_\Omega\rho_0x_N^2dx.
$$
Then, by assumption, one has $M_0>0$.
By Proposition \ref{prop1} and the boundary condition (\ref{BC}), we have
\begin{equation}\label{NEW1}
\rho|_{\partial B_{R_0}^+}=p|_{\partial B_{R_0}^+}=0,\quad u|_{\partial\mathbb R^N_+}=u|_{\partial B_{R_0}^+}=0,\quad\nabla u|_{\partial B_{R_0}^+}=0.
\end{equation}

Using equation (\ref{CNS1}), it follows from integration by parts that
\begin{align}\label{NEW2}
  \frac{d}{dt}\int_{B_{R_0}^+}\rho dx=\int_{B_{R_0}^+}\partial_t\rho dx= -\int_{B_{R_0}^+}\text{div}(\rho u)dx=0,
\end{align}
and
\begin{align}\label{NEW3}
  \frac{d}{dt}\int_{B_{R_0}^+}\rho x_N^2 dx=&\int_{B_{R_0}^+}\partial_t\rho x_N^2dx= -\int_{B_{R_0}^+}\text{div}(\rho u)x_N^2dx\nonumber\\
  =&2\int_{B_{R_0}^+}\rho u_Nx_Ndx.
\end{align}
Recalling the boundary conditions (\ref{NEW1}), and using equation (\ref{CNS2}), it follows from integration by parts that
\begin{align}
  \frac{d}{dt}\int_{B_{R_0}^+}&\rho u_Nx_N dx=\int_{B_{R_0}^+}\partial_t(\rho u_N)x_Ndx\nonumber\\
  =&-\int_{B_{R_0}^+}[\text{div}(\rho u_Nu)-\mu\Delta u_N-(\mu+\lambda)\partial_N\text{div}\,u+\partial_Np]x_N dx\nonumber \\
  =&\int_{B_{R_0}^+}(\rho|u_N|^2-\mu\partial_Nu_N-(\mu+\lambda)\text{div}\,u +p)dx\nonumber\\
  =&\int_{B_{R_0}^+}(\rho|u_N|^2+p)dx\geq\int_{B_{R_0}^+}pdx. \label{NEW4}
\end{align}

By (\ref{NEW2}) and the assumption, we have
\begin{equation}\label{NEW5}
\int_{B_{R_0}^+}\rho(x,t)dx=\int_{B_{R_0}^+}\rho_0 dx=M_0.
\end{equation}
Following the same argument as that for (\ref{ESTP}),
one can obtain
\begin{align}\label{NEW6}
  \int_{B_{R_0}^+}p(x,t)dx\geq Ae^{\underline s_0/c_v}\left(\frac{\omega_N}{2}R_0^N\right)^{1-\gamma} M_0^\gamma,
\end{align}
and thus, it follows from (\ref{NEW4}) that
\begin{align*}
  \int_{B_{R_0}^+}\rho(x,t) u_N(x,t)x_Ndx\geq& \int_{B_{R_0}^+}\rho_0 u_{0N}x_Ndx+Ae^{\underline s_0/c_v}\left(\frac{\omega_N}{2}R_0^N\right)^{1-\gamma} M_0^\gamma t\\
  =&M_1+Ae^{\underline s_0/c_v}\left(\frac{\omega_N}{2}R_0^N\right)^{1-\gamma} M_0^\gamma t,
\end{align*}
which, substituted into (\ref{NEW3}), yields
\begin{align*}
  \int_{B_{R_0}^+}\rho(x,t)x_N^2dx\geq& 2\int_{B_{R_0}^+}\rho_0x_N^2dx +2M_1t+Ae^{\underline s_0/c_v}\left(\frac{\omega_N}{2}R_0^N\right)^{1-\gamma} M_0^\gamma t^2\\
  =&2M_2+2M_1t+Ae^{\underline s_0/c_v}\left(\frac{\omega_N}{2}R_0^N\right)^{1-\gamma} M_0^\gamma t^2.
\end{align*}
On the other hand, recalling (\ref{NEW6}), we have
$$
\int_{B_{R_0}^+}\rho(x,t)x_N^2dx\leq R_0^2M_0.
$$
Combining the above two estimates, and recalling that $M_0>0$, one obtains $t\leq T_{**}$, for some positive
time $T_{**}$. This completes the proof of Case II.
\end{proof}

\begin{proof}[\emph{\textbf{Proof of Theorem \ref{thm2}}}]
Note that for $N=1$, the cases $\Omega=\mathbb R^N\setminus \overline B_{r_0}$ and $\Omega=\mathbb R^N_+$ are essentially
the same, because the domain $\mathbb R\setminus[-r_0, r_0]$ breaks into
two half lines, and system (\ref{CNS1})--(\ref{CNS3}) on these two half lines does not effect each other. Therefore, we only need to consider the case $\Omega=\mathbb R_+$. Denote
$$
K_0=\int_0^\infty\rho_0dx,\quad K_1=\int_0^\infty\rho_0u_0xdx,\quad K_3=\int_0^\infty\rho x^2 dx.
$$
Then, by assumption, we have $K_0>0$.
By Proposition \ref{prop1} and the boundary condition (\ref{BC}), we have
$$
\rho|_{x=R_0}=p|_{x=R_0}=u|_{x=0}=u|_{x=R_0}=\partial_x u|_{x=R_0}=0.
$$

Following the arguments in Case II of the proof of Theorem \ref{thm1}, we have
\begin{equation}\label{ADD1}
\int_0^{R_0}\rho(x,t)dx=\int_0^{R_0}\rho_0dx=K_0,
\end{equation}
and
\begin{align}
\frac{d}{dt}\int_0^{R_0}\rho x^2dx=&2\int_0^{R_0}\rho uxdx,\label{ADD2}\\
  \frac{d}{dt}\int_0^{R_0}\rho uxdx=&\int_0^{R_0}(\rho u ^2+p)dx. \label{ADD3}
\end{align}
Integrating equation (\ref{CNS3}) over $(0,R_0)\times(0,t)$ yields
\begin{equation}
  \int_0^{R_0}\rho(x,t) E(x,t)dx=\int_0^{R_0}\rho_0 E_0dx=\int_0^\infty \rho_0\left(\frac{|u_0|^2}{2}+c_v\theta_0\right)dx=:\mathcal E_0. \label{ADD4}
\end{equation}
Note that, by the assumption $\int_0^\infty\rho_0dx>0$ and using the compatibility condition (\ref{COMP}), one has $\mathcal E_0>0$.

If $R\geq2c_v$, i.e.\,$\gamma\geq3$, then
\begin{align*}
  \rho u ^2+p=&\rho u^2+R\rho\theta=2\rho\left(\frac{u^2}{2}+c_v\theta\right)+(R-2c_v)\rho\theta\\
  =&2\rho E+(R-2c_v)\rho\theta\geq2\rho E,
\end{align*}
and if $R<2c_v$, i.e.\,$\gamma\in(1,3)$, then
\begin{align*}
  \rho u^2+p=&\frac{R}{c_v}\rho\left(\frac{u^2}{2}+c_v\theta\right) +\left(1-\frac{R}{2c_v}\right) \rho u^2\\
  =&(\gamma-1)\rho E+\left(1-\frac{R}{2c_v}\right) \rho u^2\geq(\gamma-1)\rho E.
\end{align*}
Therefore, for any $\gamma\in(1,\infty)$, we have
$$
\int_0^{R_0}(\rho u^2+p)dx\geq\min\{2,\gamma-1\}\int_0^{R_0}\rho Edx.
$$

Thanks to the above, and recalling (\ref{ADD4}), we then obtain
$$
\int_0^{R_0}(\rho u^2+p)dx\geq\mathcal \min\{2,\gamma-1\}\mathcal E_0.
$$
Substituting this into (\ref{ADD3}), and integrating in $t$ yields
\begin{align*}
  \int_0^{R_0}\rho(x,t)u(x,t)xdx\geq&\int_0^{R_0}\rho_0u_0xdx+\min\{2,\gamma-1\} \mathcal E_0t\\
  =&K_1+\min\{2,\gamma-1\} \mathcal E_0t.
\end{align*}
With the aid of the above estimate, it following from (\ref{ADD2}) that
\begin{align*}
  \int_0^{R_0}\rho x^2dx\geq&\int_0^{R_0}\rho_0x^2dx+2K_1t+\min\{2,\gamma-1\} \mathcal E_0t^2\\
  =&K_2+2K_1t+\min\{2,\gamma-1\} \mathcal E_0t^2.
\end{align*}
On the other hand, recalling (\ref{ADD1}), we have
$$
\int_0^{R_0}\rho(x,t)x^2dx\leq R_0^2\int_0^{R_0}\rho(x,t)dx=K_0R_0^2.
$$
Combing the above two estimates, and recalling that $\mathcal E_0>0$, one obtains $t\leq T_{***}$, for some positive time $T_{***}$. This proves the conclusion.
\end{proof}

\section*{Acknowledgments}
{D.Bian was partially supported by NSFC under the contracts 11501028 and 11871005. J.Li was partly supported by start-up fund 550-8S0315 of the South China Normal University, NSFC 11771156, and the Hong Kong RGC Grant
CUHK-14302917.}
\par

\end{document}